\documentclass[12pt,twoside]{amsart}

\usepackage{amssymb,latexsym}

\usepackage{hyperref}
\hypersetup{
  colorlinks   = true, 
  urlcolor     = blue, 
  linkcolor    = blue, 
  citecolor   = red 
}
\usepackage{anysize}
\marginsize{2cm}{2cm}{1.5cm}{1.5cm}

\newtheorem{theorem}{Theorem}[section]
\newtheorem*{theorem*}{Theorem}
\newtheorem{proposition}[theorem]{Proposition}
\newtheorem{lemma}[theorem]{Lemma}

\newtheorem{conjecture}[theorem]{Conjecture}
\newtheorem{corollary}[theorem]{Corollary}

\theoremstyle{definition}

\theoremstyle{remark}
\newtheorem{remark}[theorem]{Remark}

\numberwithin{equation}{section}

\DeclareMathOperator{\vol}{vol}
\DeclareMathOperator{\area}{area}

\newcommand{\const}{{\rm const}}

\renewcommand{\epsilon}{\varepsilon}
\renewcommand{\phi}{\varphi}
\renewcommand{\kappa}{\varkappa}

\begin{document}

\title{Mahler's conjecture for some hyperplane sections}

\author{Roman Karasev}

\address{Roman Karasev, Dept. of Mathematics, Moscow Institute of Physics and Technology, Institutskiy per. 9, Dolgoprudny, Russia 141700}
\address{Roman Karasev, Institute for Information Transmission Problems RAS, Bolshoy Karetny per. 19, Moscow, Russia 127994}

\email{r\_n\_karasev@mail.ru}
\urladdr{http://www.rkarasev.ru/en/}

\thanks{Supported by the Federal professorship program grant 1.456.2016/1.4, by the Russian Foundation for Basic Research grants 18-01-00036 and 19-01-00169}

\keywords{Mahler's conjecture, symplectic volume, symplectic reduction, closed characteristics}
\subjclass[2010]{52B60, 37J45, 53D20}

\begin{abstract}
We use symplectic techniques to obtain partial results on Mahler's conjecture about the product of the volume of a convex body and the volume of its polar. We confirm the conjecture for hyperplane sections or projections of $\ell_p$-balls or the Hanner polytopes.
\end{abstract}

\maketitle

\section{Introduction}

In 1939 Mahler conjectured~\cite{ma1939} that for every centrally symmetric convex body $K\subset\mathbb R^n$ and its polar $K^\circ$ the inequality 
\[
\vol K\cdot \vol K^\circ \ge \frac{4^n}{n!}
\]
holds. Mahler has established the conjecture for $n=2$ himself, the case $n=3$ was done in the recent paper \cite{iriyeh2017}, whose short and clear exposition is \cite{fhmrz2019}.

The best result in arbitrary dimension is with $\frac{\pi^n}{n!}$ on the right hand side in \cite{ku2008}. Mahler also had a conjecture for bodies $K$ that are not necessarily centrally symmetric, but we limit ourselves to the symmetric case here, because the symplectic approach we use seems to have nothing to say about the non-symmetric Mahler conjecture, see \cite[Remark after Theorem 4.1]{abksh2014}.

It has been long known (as the Blaschke--Santal\'o inequality) that the maximum of the volume product $\vol K\cdot \vol K^\circ$ among centrally symmetric bodies is attained at ellipsoids, linear images of the unit ball, see \cite[Chapter 9]{gruber2007}. An equality case of Mahler's conjecture, where the volume product presumably attains its minimum, is, for example, when $K$ is a cube or its polar, the cross-polytope; this is easy to check by direct calculation. There also exist other conjectural minimizers, the Hanner polytopes, which by definition are the centrally symmetric polytopes that can be obtained from segments by repeatedly applying one of the two following operations: taking the Cartesian product, or taking the $\ell_1$-sum, the convex hull of the union of two polytopes in orthogonal linear subspaces. The cube and the cross-polytope are particular Hanner polytopes, and in \cite{nazarov2010,kim2014} it was shown that Mahler's conjecture holds for all convex bodies sufficiently close to a given Hanner polytope, in other words, the Hanner polytopes are indeed local minima of the volume product.

The described above results were obtained using different tools of convex and discrete geometry. In~\cite{aaok2013} it was proposed to use the symplectic point of view on this problem, in particular it was shown that Mahler's conjecture reduces to Viterbo's conjecture~\cite{vit2000} in symplectic geometry, whose statement is
\[
\vol S \ge \frac{c_{EHZ}(S)^n}{n!}
\]
for any convex $S\subset\mathbb R^{2n}$. Here $c_{EHZ}$ is the Ekeland--Hofer--Zehnder capacity, a somewhat mysterious symplectic invariant, which has an interpretation in terms of the smallest action closed characteristics of the hypersurface $\partial S$. We recommend the textbook \cite{hofer1994} as further reading about symplectic capacities, although we give relevant definitions wherever we use the symplectic notions in this paper.

In the particular case of Mahler's conjecture, $S$ is a \emph{Lagrangian product} (the product of a convex body in $\mathbb R^n$ and another convex body in its dual $\mathbb R^n$), $K\times K^\circ$, and the Ekeland--Hofer--Zehnder capacity of $S$ is the shortest length of a closed billiard trajectory in $K$ with length measured in the norm with unit ball $K$, these facts were established in \cite{aao2012}. In accordance with Mahler's conjecture, it was shown in~\cite{aaok2013} that $c_{EHZ}(K\times K^\circ) = 4$ for any convex and centrally symmetric $K\subset\mathbb R^n$.

In this paper we do not use Viterbo's conjecture, but we utilize somewhat simpler symplectic arguments to establish certain particular cases of Mahler's conjecture. Viterbo's conjecture was formulated by assuming that the optimal body is the standard ball or its convex image under a symplectomorphism. This is a much better (conjectural) description of the set of optimal bodies than what we have in Mahler's conjecture. Although not using Viterbo's conjecture, we show in this paper that the usage of symplectic balls indeed helps to prove something.

The results of this paper can be summarized as (the union of Theorems \ref{theorem:mahler-lph} and \ref{theorem:mahler-hh} below):

\begin{theorem*}
Mahler's conjecture holds for hyperplane sections of $\ell_p$-balls $(1\le p\le +\infty)$ and Hanner polytopes.
\end{theorem*}

Mahler's conjecture is invariant under linear transformations of $K$ and corresponding inverse transpose linear transformations of $K^\circ$. The conjecture is also invariant with respect to interchanging $K$ and $K^\circ$. Hence we obtain that Mahler's conjecture also holds for not necessarily orthogonal projections of $\ell_p$-balls and Hanner polytopes to hyperplanes. The case of hyperplane sections of the cube in Mahler's conjecture has attracted some attention \cite{ball1995,lopezreisner1998} itself and seems to have not been resolved before; although in \cite{lopezreisner1998} it was verified in dimensions up to $9$. The cited results were obtained with methods very different from ours. 

It is also worth noting that in \cite{meyerreisner2006} the non-symmetric version of Mahler's conjecture was verified for sections and projections of a simplex of codimension 1 and 2. As it is noted above, the non-symmetric case so far resists any sort of symplectic approach.

In the following sections we introduce the \emph{symplectic reduction approach} to Mahler's conjecture and show that it indeed works in the case of one-dimensional reduction of a symplectic ball or its slight generalization, that simply means a hyperplane section in the statement of the main theorem. Additionally, in Appendix \ref{section:capacity} we show that the Ekeland--Hofer--Zehnder capacity of centrally symmetric convex bodies does not decrease under linear symplectic reductions that we use in our approach to Mahler's conjecture, thus hinting that the symplectic reduction approach may be promising in resolving the conjecture in full generality.

\subsection*{Acknowledgments}
The author thanks Arseniy Akopyan for suggestions and corrections, Felix Schlenk for numerous useful remarks, Shlomo Reisner for remarks on previous work on the subject, and the unknown referee for numerous useful remarks and corrections.

\section{Symplectic reduction in Mahler's conjecture}

\subsection{Producing convex bodies as projections of a high-dimensional cross-polytope}

Recall that any centrally symmetric convex body $K\subset\mathbb R^n$ can be approximated in the Hausdorff metric by linear images of cross-polytopes $C\subset\mathbb R^N$, the polar bodies of cubes. For this, it is sufficient to take a dense set if pairs $\{x_i, -x_i\}_{i=1}^N$ in $\partial K$ and consider the linear map $f:\mathbb R^N\to \mathbb R^n$ that takes every basis vector $e_i\in\mathbb R^N$ to its corresponding $x_i$. Since the unit cross-polytope $C$ is the convex hull of $\{e_i,-e_i\}_{i=1}^N$ by definition, $f(C)$ is contained in $K$ and $\varepsilon$-approximates $K$ if $\{x_i, -x_i\}_{i=1}^N$ is an $\varepsilon$-net of $\partial K$.

Since Mahler's conjecture itself is invariant under linear transformations, we may assume that such a linear image is an image of an orthogonal projection of a unit cross-polytope along some linear subspace $L\subset\mathbb R^N$. What happens to $K^\circ$ then? In fact, $K^\circ$ is then approximated by the section $C^\circ\cap L^\perp = (C/L)^\circ$, where $L^\perp\subset\mathbb R^N$ is the subspace orthogonal to $L$ and $C^\circ$ is the unit cube $[-1,1]^N$.

\subsection{Basics of symplectic geometry and symplectic reduction}

Let us translate the above picture to symplectic terms, showing what is happening with the body $S = C\times C^\circ$ when it produces the body $K\times K^\circ$, whose volume is the object of Mahler's conjecture. We first recall some basic definitions of symplectic geometry and refer the reader to the textbook \cite{cannasdasilva2008} for a detailed exposition. Let $q_1,\ldots, q_N$ be coordinates in $\mathbb R^N$ and $p_1,\ldots, p_N$ be the coordinates in its dual $\mathbb R^N$. The Cartesian product of $\mathbb R^N$ and its dual naturally carries a skew-symmetric non-degenerate bilinear form
\[
\omega( (p',q'), (p'',q'') ) = \sum_{i=1}^N \left( p_i'q''_i - p''_i q'_i \right),
\] 
which may be also written as 
\begin{equation}
\label{equation:standard-structure}
\omega = \sum_{i=1}^N dp_i\wedge dq_i
\end{equation}
in the notation of differential geometry. This $\omega$ is invariant when we linearly transform $\mathbb R^N$ and at the same time apply the inverse transpose linear transformation to the dual $\mathbb R^N$. The form $\omega$ vanishes on the $p$-subspace $\mathbb R^N\subset\mathbb R^{2N}$ and the $q$-subspace $\mathbb R^N\subset \mathbb R^{2N}$, hence those spaces are \emph{isotropic} with respect to $\omega$. They are also \emph{Lagrangian}, since they are maximal by inclusion among isotropic subspaces of $\mathbb R^{2N}$.

More generally, whenever a smooth manifold $M$ carries a two-form $\omega\in\Omega^2(M)$, which is closed, $d\omega = 0$, and is non-degenerate at every point $p\in M$ (that is, induces a non-degenerate bilinear form on the tangent space $T_pM$), we call $\omega$ a \emph{symplectic structure} on $M$. The Darboux theorem \cite[Theorem 8.1]{cannasdasilva2008} asserts that for every point $p\in M$ there exists a coordinate chart in a neighborhood of $p$, where $\omega$ has precisely the same form as in \eqref{equation:standard-structure}, with $2N$ equal to the dimension of $M$. A submanifold $L\subset M$ is then called \emph{isotropic}, is $\omega$ vanishes on $L$, and is called \emph{coisotropic}, if at every point $p\in L$ the tangent space $T_pL$ contains its $\omega$-orthogonal complement in $T_pM$. The passage to $\omega$-orthogonal subspace, for linear subspaces of $\mathbb R^{2N}$, interchanges isotropic and coisotropic subspaces.

We want to restate the section and projection construction for the product of a convex body and its polar in $\mathbb R^{2N}$ in symplectic terms. We take a linear subspace $L\subset\mathbb R^{2N}$, contained in $\mathbb R^N\subset\mathbb R^{2N}$ of $q$-coordinates and hence isotropic. We also take the orthogonal complement of $L$ with respect to the symplectic form $\omega$, the coisotropic subspace $L^\omega$. Since ``isotropic'' means that the restriction of $\omega$ to $L$ is zero, $L\subset L^\omega$ and $L^\omega$ is indeed a coisotropic subspace. Now the procedure to obtain $S'=K\times K^\circ$ from $S=C\times C^\circ$ is generalized as follows: We take the intersection of $S$ with $L^\omega$ and then take the projection along $L$:
\[
S' = (S\cap L^\omega)/L.
\]
The \emph{projection along $L$} is the linear quotient map $\mathbb R^{2N} \to \mathbb R^{2N}/L$, which we restrict to $L^\omega$. 

This construction is close to the notion of symplectic reduction, so let us also call this process \emph{reduction of $S$ along $L$}, see \cite[Chapter 24]{cannasdasilva2008}. In our case, we take the linear Hamiltonians $H_1,\ldots, H_{N-n}$ so that $L^\omega$ is the solution set of the system of equations
\begin{equation}
\label{equation:hamiltonians-zero}
H_1 = \dots = H_{N-n} = 0.
\end{equation}
We also take their respective Hamiltonian vector fields $X_1,\ldots, X_{N-n}$, defined by the identities
\[
\omega (X_i, Y) = - dH_i(Y)
\]
for any vector $Y$. Those vector fields are constant vectors in our case and they span the original subspace $L$. The vector fields $X_i$ Lie-commute, $[X_i, X_j]=0$ for any $i$ and $j$, since they are constant. Moreover, the Hamiltonians $H_i$ Poisson-commute, $\{H_i, H_j\} = 0$ for any $i$ and $j$, since they only depend on $p$ coordinates and the Poisson bracket is given in coordinates as
\[
\{F,G\} = \sum_{i=1}^N \left( \frac{\partial F_i}{\partial q_i}\frac{\partial G_i}{\partial p_i} - \frac{\partial G_i}{\partial q_i}\frac{\partial F_i}{\partial p_i} \right).
\]
Recall also that if functions $F$ and $G$ are considered as Hamiltonians and their corresponding Hamiltonian vector fields are $X_F$ and $X_G$ then 
\[
\{F,G\} = X_F(G) = - X_G(F).
\]
To summarize, we deal with a symplectic reduction, since we first solve the system of equations \eqref{equation:hamiltonians-zero} and then take the quotient along the group action generated by the flows of the corresponding Hamiltonian vector fields $X_1,\ldots, X_{N-n}$.

The symplectic volume form in dimension $2N$ is defined as $\frac{\omega^N}{N!}$. In Darboux coordinates it just equals
\[
\frac{\omega^N}{N!} = dp_1\wedge dq_1\wedge \dots \wedge dp_N\wedge dq_N,
\]
which corresponds up to sign with the standard volume in $\mathbb R^{2N}$, for example. Hence we may speak about estimating the symplectic volume of the product $K\times K^\circ$ from below in Mahler's conjecture.

\subsection{Linear reduction of a symplectic ball}
\label{section:linear-reduction-ball}
Let us test this kind of linear reduction on the body $S=B^{2N}\subset\mathbb R^{2N}$, the standard symplectic ball, given by
\[
\sum_{i=1}^N p_i^2 + q_i^2 \le 1 
\]
in Darboux coordinates of $\mathbb R^{2N}$. The unit ball is not expressed as $K\times K^\circ$ and therefore the linear reduction of the ball is not directly related to Mahler's conjecture.

Linear transformations preserving the ball $B^{2N}$ and the form $\omega$ are just the unitary group, which may be considered as thew group preserving the Hermitian form $\sigma+i\omega$, where $\sigma$ is the symmetric bilinear form corresponding to the ball and $\omega$ is the skew-symmetric form of the symplectic structure.

This means that any linear Lagrangian subspace of $\mathbb R^{2N}$ has all Hermitian products of its vectors real and may be unitarily (and therefore symplectically) transformed into the standard $q$-subspace $\mathbb R^N\subset\mathbb R^{2N}$, keeping the ball invariant. More generally, any isotropic subspace $L\subset\mathbb R^{2N}$ can also be symplectically transformed to the subspace with coordinates $q_1,\ldots, q_{N-n}$ arbitrary and all other coordinates zero and the ball will remain invariant. The coisotropic linear subspace $L^\omega$ is then defined by the equations
\[
p_1 = \dots = p_{N-n} = 0,
\]
and the reduction $(B^{2N}\cap L^\omega)/L$ is again a symplectic unit ball of dimension $2n$. 

We summarize that \emph{the linear reduction makes a ball of $2n$-volume $\frac{\pi^n}{n!}$ from the ball of $2N$-volume $\frac{\pi^N}{N!}$, which is in accordance (after scaling) with what we want to have with $C\times C^\circ$ and $K\times K^\circ$, i.e. to make $\frac{4^n}{n!}$ out of $\frac{4^N}{N!}$.}

\subsection{Nonlinear symplectomorphic images of the ball}

Now recall the fact that $C\times C^\circ$ is symplectically a ball of radius $\sqrt{4/\pi}$ in a certain sense. In fact we only need that it can be approximated in the Hausdorff metric by symplectomorphic images of balls of radius tending to $\sqrt{4/\pi}$, this is discussed in Section \ref{section:balls} below, our construction essentially uses the ideas and pictures from~\cite[Chapter 3]{schlenk2005}. 

In symplectic terms, we are essentially studying the following question (after rescaling to get rid of the multiplier $\sqrt{4/\pi}$): A symplectomorphism $\phi :\mathbb R^{2N} \to \mathbb R^{2N}$ sends $L^\omega$ to a coisotropic submanifold $M = \phi(L^\omega)\subset \mathbb R^{2N}$ and sends $C\times C^\circ$ to an approximate ball, which we may scale to $B^{2N}$. We are trying to understand the symplectic volume of the intersection $B^{2N}\cap M$ after taking its quotient along the foliation into isotropic fibers
\[
\mathcal F = \{\phi(L+t)\}_{t\in L^\omega}.
\]

Eventually, we need to show that
\[
\vol \left( (B^{2N}\cap M)/\mathcal F \right) \ge \frac{\pi^n}{n!}.
\]

Let us again state the problem in more symplectic terms. We have $N-n$ smooth pairwise Poisson-commuting functions $H_1,\ldots, H_{N-n}$, which in our construction are odd functions without critical points. In this setting the manifold $M$ is given by
\[
M = \{x : H_1(x) = \dots = H_{N-n}(x) = 0\}.
\]
We consider their respective Hamiltonian vector fields $X_1,\ldots, X_{N-n}$ and the foliation $\mathcal F$ of $M$ obtained by integrating these pairwise Lie-commuting vector fields. Then we take the image of $B^{2N}\cap M$ in the quotient of $\mathbb R^{2N}$ by this foliation. This restatement with several Hamiltonian functions and vector fields makes us guess that Mahler's conjecture might be accessible by induction on the number of functions with the induction step given by the following conjecture (or a version of it):

\begin{conjecture}
Assume $B^{2N}\subset \mathbb R^{2N}$ is the standard ball and $H : \mathbb R^{2N}\to \mathbb R$ is an odd smooth function without critical points, with Hamiltonian vector field $X_H$ and foliation into its trajectories $\mathcal F$. Then the reduction 
\[
(B^{2N}\cap \{H=0\})/\mathcal F
\]
contains a symplectomorphic image $\phi(B^{2N-2})$ with an odd smooth symplectomorphism $\phi$.
\end{conjecture}

\subsection{A similar estimate for the Riemannian volume}

We guess that the symplectic reduction construction makes sense because this symplectic construction has a simpler Riemannian version with a certain volume estimate. From the construction in Section~\ref{section:balls} it is clear that the map $\phi : \mathbb R^{2N}\to\mathbb R^{2N}$ can be chosen to be odd, $\phi(-x) = -\phi(x)$, and the manifold $M$ is then centrally symmetric around the origin. Then the Borsuk--Ulam theorem, applied to the odd map $\phi^{-1}$, asserts that for any radius $r$ the set $S^{2N-1}(r)\cap M$ of dimension $N+n-1$ intersects every $(N-n)$-dimensional equatorial subsphere $\Sigma \subset S^{2N-1}(r)$ at least twice, see the details in \cite[Section 2]{ahk2016} (the idea essentially goes back to \cite{barlov1982}), where this idea produces another proof of Vaaler's theorem on sections of the cube.

These Borsuk--Ulam type considerations are sufficient to invoke Crofton's formula and conclude that the $(N+n)$-dimensional Riemannian volume of $B^{2N}\cap M$ is at least $\frac{\pi^{\frac{N+n}{2}}}{\frac{N+n}{2}!}$. This argument is an elementary case of Gromov's ``waist of the sphere'' theorem \cite{grom2003}. Of course, the $(N+n)$-dimensional Riemannian volume of $B^{2N}\cap M$ is not the same as the symplectic volume of this manifold. Moreover, we also have to take the quotient of this manifold by the isotropic foliation $\mathcal F$ in order to make its symplectic volume meaningful. What is possible to obtain from a Crofton-type argument in a particular case is given around Lemma \ref{lemma:codim2} below.

\section{Convex symplectic balls inside some Lagrangian products}
\label{section:balls}

In the previous section we have found some hints that certain symplectic reductions of a symplectic ball behave well in terms of the volume of the reduction. Now we are going to remind the technique of~\cite{schlenk2005} that allows to show that the product $C\times C^\circ$ for $C=[-1,1]^N$ can in fact be approximated by symplectomorphic images of the standard ball with arbitrary precision. 

Here we are going to prove a slight generalization of the mentioned fact about symplectic balls and $C\times C^\circ$. We  give some freedom and approximate the $C\times C^\circ$ by a symplectic image of $B^{2N}(R)$ for $R$ arbitrarily close to $\sqrt{4/\pi}$, but not the precise $R=\sqrt{4/\pi}$. This does not affect the application to Mahler's conjecture since we are free to pass to the limit $R\to\sqrt{4/\pi}$. Next, we consider a convex body $K\subset\mathbb R^N$ equal to the unit ball of the $\ell_p$ norm (with $1<p<+\infty$) and its product $K\times K^\circ$; $C\times C^\circ$ is a limit case $p=+\infty$ of such products. Again, in the application to Mahler's conjecture we are free to pass to the limit.

\begin{proposition}
\label{proposition:convex-lp}
If $K\subset \mathbb R^N$ is the unit ball of an $\ell_p$ norm $(1 < p < +\infty)$, then $K\times K^\circ$ contains a symplectomorphic image $\phi(B^{2N}(R))$ of the ball $B^{2N}(R)$ for $R$ arbitrarily close to $\sqrt{4/\pi}$. Moreover, all images $\phi(B^{2N}(r))$ for $0<r\le R$ may be assumed strictly convex and $\phi$ may be chosen to be odd, $\phi(-x) = -\phi(x)$.
\end{proposition}


\begin{proof} 
We mostly repeat the argument of~\cite{schlenk2005}, see also~\cite{lmdsch2013}, with slight modifications. To avoid confusion with the symplectic coordinates $p$ and $q$ we denote the exponents by $\alpha$ and $\beta$ instead, so that 
\[
\frac{1}{\alpha} + \frac{1}{\beta} = 1.
\]
Now we want to show that the set given by the inequalities
\[
|q_1|^\alpha + \dots + |q_N|^\alpha \le 1, \quad |p_1|^\beta + \dots + |p_N|^\beta \le 1
\]
contains a symplectic image of a ball with radius arbitrarily close to $\sqrt{4/\pi}$.

Start with an area and orientation preserving two-dimensional diffeomorphism $f :\mathbb C\to\mathbb R^2$ ($z\mapsto (q(z), p(z)$), such that 
\begin{equation}
\label{equation:eps-to-rect}
|q(z)|^\alpha \le \frac{\pi |z|^2}{4} + \epsilon, \quad |p(z)|^\beta \le \frac{\pi |z|^2}{4} + \epsilon,
\end{equation}
where $\epsilon$ is an arbitrarily small positive number. These inequalities can be achieved by an area and orientation preserving map because they mean that the disc of radius $r$ centered at the origin has to get into a rectangle of area slightly larger than $\pi r^2$. Indeed, we have the chain of inequalities
\[
|z|\le r\Rightarrow |q|^\alpha, |p|^\beta \le \frac{\pi r^2}{4} + \epsilon\Rightarrow |p|\cdot|q| \le \left( \frac{\pi r^2}{4} + \epsilon \right)^{\frac{1}{\alpha}+\frac{1}{\beta}} = \frac{\pi r^2}{4} + \epsilon\Rightarrow 4 |p|\cdot |q| \le \pi r^2 + 4 \epsilon,
\]
which proves the consistency of the areas.

We need a map $f$ such that the function given by $F(p,q) = |f^{-1}(p,q)|^2$ (in other words, a push-forward of $|z|^2$ by $f$) be a smooth and strictly convex function with unique minimum. The existence of a map $f$ producing a function $F$ with convex sublevel sets is geometrically intuitive. But we need a stronger property than the convexity of the sublevel sets, $F$ must be a convex function itself. We need this, because we then consider a sum of such functions of different variables and want this sum (and its sublevel sets) to remain convex. Let us start with the construction, first put 
\[
G(p,q) = 4 \max\{|q|^\alpha, |p|^\beta\},
\]
this is a convex function whose sublevel sets $\{ (p,q)\in \mathbb R^2\ |\ G \le A\}$ have area $A$ for $A\ge 0$, since  
\[
G(p,q)\le A\Leftrightarrow |q|^\alpha, |p|^\beta \le \frac{A}{4}
\]
and the area of this sublevel set then equals 
\[
4 \left( \frac{A}{4} \right)^{\frac{1}{\alpha}+\frac{1}{\beta}} = 4 \frac{A}{4}= A
\]

Then we perturb the function $G$ slightly to the new function $F$, which is smooth, strictly convex, has unique minimum at the origin, and has sublevel sets $\{F\le A\}$ of area $A$. For this, we first approximate $G$ by the strictly convex functions
\[
G_N(p,q) = c_N \left( |q|^{\alpha N} + |p|^{\beta N}\right)^{1/N},
\]
where $N$ is sufficiently large and the constant $c_N$ chosen to normalize the areas of the sublevel sets. Such $G_N$ has all the required properties except for smoothness at the origin and converges to $G(p,q)$ as $N\to\infty$.

Now it remains to modify $G_N$ near the origin to make it smooth and thus obtain $F$ with the required properties. Assume we start from a neighborhood of the origin, where after rescaling of the coordinates $p$ and $q$ we have
\[
F(p,q) = c\cdot (|p|^u + |q|^v)^{1/u+1/v}
\]
for $F(p,q)\le 1$, here we put $u = \alpha N, v=\beta N$ in the beginning of the procedure and introduce a constant $c$.  This function $F$ is strictly convex in the range $1/2\le F(p,q) \le 1$, which is expressed as a strict inequality in terms of its derivatives up to second order. We want to modify it so that it remains the same at the level set $\{F(p,q)=1\}$ and is expressed by the similar formula at the level set $\{F(p,q)=1/2\}$ with different $u$ and $v$, but keeping the required properties of strict convexity and areas of sublevel sets. 

Consider $u$ and $v$ as not constants, but slightly varying functions with sufficiently small first and second derivatives. The convexity of $F$, as expressed in terms of its second derivatives, will be preserved if the first and second derivatives of $u$ and $v$ are kept sufficiently small in the required range. The requirement that $\area\{F\le A\} = A$ can also be kept by considering the constant $c$ also varying with $p$ and $q$, again, if the first and second derivatives of $u$ and $v$ are kept small then the first and second derivatives of the coefficient $c$ will also be kept small, not violating the convexity of $F$. 

Hence there exists a small neighborhood of the parameter pair $(u,v)$ such that for any $(u',v')$ in this neighborhood we can modify $F$ in the set $\{F(p,q)<1\}$ so that it is expressed by the same formula with new parameters $u',v'$ at the level set $\{F(p,q)=1/2\}$ keeping its required properties. After such a step we may rescale the coordinates and repeat the procedure. We aim at the pair of parameters $(u,v)=(2,2)$. From compactness considerations it is indeed possible to reach this value in a finite number of steps. Thus constructed function will be just $\pi(p^2+q^2)$ at a neighborhood of the origin. 

Note that thus constructed $F$ is not infinitely smooth because of using $|p|$ and $|q|$, but keeping $u,v\ge 2$ ensures that it has continuous second derivatives at least. After that it is possible to make it infinitely smooth by approximating it together with its first and second derivatives by a sequence of infinitely smooth functions $(F_n)$, the strict convexity assumption, expressed in terms of second derivatives, will be satisfied for sufficiently close approximation. The assumption $\area \{F_n\le A\} = A$ will be met, if we modify $F_n$ by a factor function $c_n(F_n(p,q))$, whose first and second derivatives will also tend to zero as $n\to\infty$, not spoiling the convexity of $c_nF_n$ for sufficiently large $n$. Eventually, for sufficiently large $n$ the infinitely smooth function $c_nF_n$ will also have the required properties and may be chosen as our final $F$.

After this, it remains to design an area-preserving diffeomorphism $f$ that transforms $|z|^2$ to $F$, which is possible because it only requires the assumption $\area \{F \le A\} = A$ for $A\ge 0$ and the good structure of $F$ near the origin. Near the origin $f$ may be chosen linear, from our construction of $F$, and it is possible to have $f$ odd in this setting, because $F$ we may assume that $F$ was constructed even.

Now the Cartesian product $f^{\times N}$ transforms the ball $B^{2N}(R)$ to the set defined by the equation
\[
F(p_1, q_1) + \dots + F(p_N, q_N) \le R^2,
\]
from the smoothness and strict convexity of $F$ it follows that this set is smooth and strictly convex as well. If $F$ does not deviate much from $G$ and satisfies (\ref{equation:eps-to-rect}) then we have
\[
\sum_i |q_i(z)|^\alpha \le \frac{\pi |z|^2}{4} + n\epsilon,\quad \sum_i |p_i(z)|^\beta \le \frac{\pi |z|^2}{4} + n\epsilon.
\]
This means that the image of the ball of radius $\sqrt{ \frac{4}{\pi} (1 - n\epsilon)}$ fits into the product of the unit ball of $\ell_\alpha$ norm and the unit ball of $\ell_\beta$ norm, which completes the proof.
\end{proof}

\begin{remark}
\label{remark:hamiltonian}
In the above construction $\phi$ can be assumed to be a linear symplectomorphism in a small neighborhood of the origin. Hence it can be connected to a linear symplectomorphism by the smooth family of symplectomorphisms
\[
h_t(z) = \frac{1}{t}\phi(tz),
\]
and then to the identity by a family of linear symplectomorphisms. Therefore $\phi$ is smoothly isotopic to the identity through symplectomorphisms and is a Hamiltonian symplectomorphism, that is a symplectomorphism given by integration of a time-dependent Hamiltonian vector field of the form $X_t = \frac{\partial h_t(z)}{\partial t}$. Here we use that in $\mathbb R^{2n}$ any vector field $X$ preserving the symplectic structure $\omega$ (that is, $L_X\omega = 0$) is a Hamiltonian vector field.
\end{remark}

\begin{remark}
The referee asked if the construction of this section passes to Orlicz spaces, which generalize $\ell_p$ spaces in a certain way. We have no answer to this questions, but think it may be interesting.
\end{remark}

\section{Reduction by one dimension}

\subsection{Using a Crofton-type argument}

Now we return to applying the symplectic reduction to Mahler's problem. We consider $C\times C^\circ$, where $C$ is the unit cross-polytope and $C^\circ$ is the unit cube, or slightly more generally, $K\times K^\circ$, where $K$ is an $\ell_p$ ball and $K^\circ$ is its dual $\ell_q$ ball. The product of the cross-polytope and the cube is a limit case of such $K\times K^\circ$ when $p\to 1$.

Proposition~\ref{proposition:convex-lp} gives us a function $F$ on $\mathbb R^{2N}$ which is smooth, even, strictly convex, having unique minimum at the origin, and whose sublevel set $\{F \le \pi R^2\}$ is a symplectic ball (of radius $R$ before the symplectic transformation) for every $R$ and lies in $K\times K^\circ$ for $R<\sqrt{\frac{4}{\pi}} - \varepsilon$ ($\varepsilon$ will tend to $0$ once we need it). This gives a suitable approximation of $K\times K^\circ$ with symplectic balls.

In order to find the volume of the symplectic reduction of the ball we may build a section of the symplectic reduction map. Generally, our use of Proposition~\ref{proposition:convex-lp} allows us to conclude that, when we reduce to $L^\omega/L$, the sets $\phi(B^{2N}(R))\cap L^\omega$ (here $R$ always denotes some radius less than $\sqrt{\frac{4}{\pi}}-\varepsilon$) are all strictly convex and smooth bodies, and their sections by $L + t$ are also smooth and strictly convex bodies of dimension $N-n$, or just points in the boundary case, or empty sets.

We are going to consider the case $N-n=1$, that is one Hamiltonian $H$ and one vector field $X_H$ in the symplectic description of the reduction. In this case the reduction has two natural sections, since it is the projection along the lines $L+t$ in the hyperplane $L^\omega$. In our setting, the integral curves of $X_H$ (the foliation $\mathcal F$) enter the ball $\phi(B^{2N}(R))$ once and leave it precisely once because of its convexity. Hence there is one way to choose the entry point of $\phi(B^{2N}(R))\cap (L+t)$, and the other way to choose the exit point, both giving a section of the quotient map of the reduction.

For brevity of notation, let us work in the coordinates before applying $\phi$, where $F(z)=|z|^2$, the symplectic balls in question are Euclidean balls, while $H$ is possibly non-linear. We distinguish between the cases when the integral curve of $X_H$ enters $B^{2N}(R)$ and exits it by the sign of the Poisson bracket 
\[
X_H(F) = \{H, F\} = - \{F,H\} = - X_F(H).
\]
The latter equation means that those cases are distinguished by the sign of the intersection between an oriented complex circle $C$, integral of the vector field $X_F$ in $\mathbb R^{2N}=\mathbb C^N$, and the hypersurface $\{H=0\}$, cooriented by the gradient of $H$.

Put $S^{2N-1}(R) = \partial B^{2N}(R)$ and $\Sigma = S^{2N-1}(R)\cap \{H=0\}$. In our case, with convexity assumptions, $\Sigma$ is  diffeomorphic to a $(2n-2)$-dimensional sphere. We split $\Sigma$ into two parts $\Sigma^+$ and $\Sigma^-$, depending on the sign of $X_F(H) = -X_H (F)$, which corresponds to the entry points and the exit points of the foliation of the ball. From the entry and exit description it follows that $\Sigma^+$ is another open $(2N-2)$-dimensional manifold projected diffeomorphically onto the symplectic reduction $(D^{2N}(R)\cap \{H=0\})/\mathcal F$. We use here the notation $D^{2N}(R)$ for an open ball in contrast with $B^{2N}(R)$, the closed ball.

To prove Mahler's conjecture in this particular case, it would be sufficient to show that $\Sigma^+$ with the symplectic structure $\omega$ contains a symplectic $(2N-2)$-ball of radius $R$, since the symplectic volume would then have the right estimate from below. It is not clear how to show this apart from the trivial case of the linear reduction of a ball from Section \ref{section:linear-reduction-ball}. But we can prove that in the case of reduction of a symplectic ball by one dimension the $\omega^n$-volume of $\Sigma^+$ is no less that the $\omega^n$-volume of the standard symplectic ball of radius $R$. This relies on the following Crofton-type formula:

\begin{lemma}
\label{lemma:codim2}
There exists a constant $c_N$ such that whenever $M$ is a $(2N-2)$-dimensional oriented submanifold, possibly with boundary, of the sphere $S^{2N-1}(R)$, then 
\[
\int_M \omega^{N-1} = c_N R^{2N-2} \int_{\{C\}} \#(C\cap M),
\]
where $\#(C\cap M)$ is the number of intersections of an $X_F$-integral oriented circle $C$ and $M$ counted with signs, and the integral on the right hand side is taken over all possible $C$ with respect to the unitary-invariant probability measure.
\end{lemma}

\begin{proof}
The radius $R$ just gives the scale factor so we can put $R=1$ in the proof. Then we observe that the trajectories of $X_F$ are complex circles and the quotient of $S^{2N-1}(R)$ by the foliation of the integral curves of $X_F$, the space of these circles, is then just the projective space $\mathbb CP^{N-1}$. The pullback of the Fubini--Study symplectic form $\omega_{FS}$ from $\mathbb CP^{N-1}$ (by definition) equals to the restriction of $\omega$ to the sphere $S^{2N-1}$ (maybe up to constant). 

Hence we consider the map $f : M \to \mathbb CP^{N-1}$ and integrate $f^* \omega_{FS}^{N-1}$ over $M$ in the left hand side of the required identity. This is essentially the same as integration of $\omega_{FS}^{N-1}$ over the projective space $\mathbb CP^{N-1}$, multiplied by the algebraic multiplicity of the map $f$ over a given point of $\mathbb CP^{N-1}$. Up to a measure zero set of critical values (by Sard's theorem) this algebraic multiplicity is well-defined, and $\omega_{FS}^{N-1}$ is a unitary-invariant density on the projective space, thus justifying the right hand side of the formula.
\end{proof}

Using the lemma, we obtain the following: Every $C$ intersects $\Sigma$ at least twice, here we use that $\phi$ and $H$ are odd and restrict $H$ to an odd function on the circle $C$. Positive intersections are collected on $\Sigma^+$, negative are collected on $\Sigma^-$. Hence 
\[
\int_{\Sigma^+} \omega^n = c_n R^{2n} \int_{\{C\}} \#(C\cap \Sigma^+) \ge c_n R^{2n} \int_{\{C\}} 1 = \frac{1}{2} c_n R^{2n}.
\]
In the case of linear $H$ almost every $C$ intersects $\{H=0\}$ precisely twice with opposite signs, and therefore the equality is attained. In the linear case $\Sigma^+$ will be symplectomorphic to the standard $B^{2n}(R)$ and therefore, in the non-linear case we have
\[
\int_{\Sigma^+} \omega^n \ge \int_{B^{2n}(R)} \omega^n,
\]
which establishes the desired volume estimate. We summarize the result of this section in:

\begin{theorem}
\label{theorem:mahler-lph}
Mahler's conjecture holds for hyperplane sections of $\ell_p$ balls $(1< p< +\infty)$ and their projections to hyperplanes. As a limit case, it holds for hyperplane section of a cube and respective projections of the cross-polytope to hyperplanes.
\end{theorem}

We may recognize the central hyperplane sections $K\subset \mathbb R^n$ of linear images of cubes $C\subset \mathbb R^{n+1}$ as centrally symmetric polytopes with $2n+2$ facets. Similarly, we may recognize their polars as centrally symmetric polytopes in $\mathbb R^n$ with $2n+2$ vertices. The latter case is clear since we may map the vertices of a cross-polytope $C\subset\mathbb R^{n+1}$ to the vertices of the given polytope and extend this map linearly; the former case is the polar of this. Hence we obtain:

\begin{corollary}
Mahler's conjecture in $\mathbb R^n$ holds for centrally symmetric polytopes having either $2n+2$ facets or $2n+2$ vertices.
\end{corollary}

\subsection{Using integration over not necessarily closed trajectories}
\label{section:integration}

In this section we provide another explanation of the reduction by one dimension, suitable for bodies $K\times K^\circ$, for a Hanner polytope $K$. Conjecturally, such bodies can be approximated by symplectic balls \cite[Question~5.2]{ostrover2014}, but the best we definitely know is that they have almost all characteristics on the boundary closed with the same action, see \cite{balitskiy2016}. Hence the argument from the previous section does not apply and we need another kind of argument. In fact the argument in this section equally applies to the case of the previous section. The only drawback is that this argument seems less plausible to be generalized to reductions by dimension more than one.

Let $S = K\times K^\circ$ be the product body, whose volume we assume known, and let 
\[
S' = (K/L)\times (K^\circ\cap L^\perp) = (K/L)\times (K/L)^\circ
\]
be the new product body, obtained by one-dimensional reduction along a line $L$. Let us, for a while, measure the volume of $S$ in terms of $\omega^n$ and the volume of $S'$ in terms of $\omega^{n-1}$, thus eliminating the inverse factorials in terms of Mahler's conjecture. We need to pass from $\int_S \omega^n$ to $\int_{S'} \omega^{n-1}$ somehow.

The symplectic reduction $S'$ can have different realizations in $S$, corresponding to different choices of the section of the quotient map $\mathbb R^{2n} \to \mathbb R^{2n}/L$ over $S'$. The choice done in the previous section represents $S'$ as half of the topological sphere $\partial S\cap \{H=0\}$, where $H$ is the linear function, whose zero set is $L^\omega$. The half $\Sigma^+$ of $\partial S\cap \{H=0\}$, according to reduction considerations, must be chosen so that the vector field $X_H$ points outside of $S$ in this half. This has a formulation in terms of the Poisson bracket of $H$ and the gauge function of $S$, which we call $F$ again. Another way to describe the choice of the half (in view of $X_F (H) = - X_H (F)$) is to say that the \emph{oriented characteristics} on the boundary of $S$ (that is, the trajectories of $X_F$ on the hypersurface $\{F = \const\}$) must intersect $\{H=0\}$ in the given direction, say, in the direction of increasing the linear function $H$.

Assuming the choice of $\Sigma^+$ symplectomorphic to $S'$, we choose a primitive $\lambda$ for $\omega$ ($d\lambda = \omega$) as
\[
\lambda = \frac{1}{2}\sum_{i=1}^n \left( p_idq_i - q_i dp_i\right)
\]
and look how to estimate the volume of $S$ from above knowing the volume of $\Sigma^+\subset \partial S$. Take a characteristic $\gamma : [a,b]\to \partial S$ staring at $\Sigma^+$ and having
\[
\int_\gamma \lambda \le A
\]
for a constant $A$, assume that $A$ is sufficiently large so that all considered characteristics almost cover the whole $\partial S$. This covering assumption means that the volume of $S$ has an upper bound
\begin{equation}
\label{equation:reduction-action-bound}
\int_S \omega^n = \int_{\partial S} \lambda \wedge \omega^{n-1} \le A \int_{\Sigma^+} \omega^{n-1}.
\end{equation}
The first equality here is the Stokes formula, the right hand side can be interpreted as follows. We consider the map $\psi : \Sigma^+\times [0,A] \to \partial S$, which takes $(p, t)$ as a starting point $p\in \Sigma^+$ and the parameter value, considers the trajectory $\gamma$ of $X_F$ in $\partial S$ starting at $p$ and takes the end point of this trajectory so that $\int_\gamma \lambda = t$. This endpoint is the value of $\psi(p, t)$ and our covering assumption is that $\psi$ is almost surjective. Note that the trajectories of $X_F$ lie in the kernel of $\omega|_{\partial S}$ and therefore we have
\[
\phi^* (\lambda\wedge \omega^{n-1}) = dt\wedge \omega^{n-1}
\] 
over $\Sigma^+\times [0,A]$. From this and the almost surjectivity of $\psi$ we have
\[
A \int_{\Sigma^+} \omega^{n-1} = \int_{\Sigma^+\times [0,A]} dt\wedge \omega^{n-1} \ge \int_{\partial S} \lambda \wedge \omega^{n-1},
\]
which explains the inequality in \eqref{equation:reduction-action-bound}.

The bound in \eqref{equation:reduction-action-bound} can be equally viewed as a lower bound on the volume of the symplectic reduction $S'$. Note that this volume argument is not the same as the Crofton-type argument in the previous section, this is actually a more general thing. In particular, there is no need to have a sympletic ball in $K\times K^\circ$ and therefore the whole argument becomes in fact elementary. The argument as given works clearly in the case when $\partial S$ is smooth, but it also works in the piece-wise smooth case (our particular case $S=K\times K^\circ$) when almost all (in terms of the measure) characteristics on $\partial S$ are well-defined (this is indeed so by the result of \cite{aao2012}).

More specifically, the bound \eqref{equation:reduction-action-bound} gives an estimate from below on the volume of $S'$, which can be written in standard terms (recalling that the standard volume forms are $\frac{\omega^n}{n!}$ and $\frac{\omega^{n-1}}{(n-1)!}$):
\[
\vol S' \ge \frac{n}{A} \vol S.
\]
The above assumptions on the number $A>0$ and the convex body $S$, when $S$ is a product $K\times K^\circ$, can be restated taking into account the interpretation of the action in terms of the length of a billiard trajectory from \cite{aao2012}: \emph{Of the segments of length $A$ of billiard trajectories in $K^\circ$ (with the length measured with $K^\circ$ as a unit ball) almost all intersect the hyperplane $\{H=0\}$ in the direction of increasing $H$.} 

Note that this assumption is satisfied when almost all billiard trajectories in $K^\circ$ are closed of length $A$, then almost all of them do not lie entirely in $\{H=0\}$ and therefore have to intersect this hyperplane in the positive direction at least once when running the length $A$. This property of the billiard trajectories has been established in the case when $K$ (and hence $K^\circ$) is a Hanner polytope in \cite{balitskiy2016} with the constant $A=4$. In view of the fact that Mahler's conjecture holds for the Hanner polytopes with equality and $\frac{n}{4}\cdot\frac{4^n}{n!} = \frac{4^{n-1}}{(n-1)!}$, we obtain the main result of this section:

\begin{theorem}
\label{theorem:mahler-hh}
Mahler's conjecture holds for hyperplane sections of Hanner polytopes and their projections to hyperplanes.
\end{theorem}

Another way to justify the usage of somewhat smooth arguments in the case, when $S=K\times K^\circ$ is apparently not smooth, we note the following. In our particular case, $S$ is a polytope, whose characteristics are well-defined on its facets and almost all of them pass from one facet to another in a well-defined way, being closed with action $A$. This is sufficient to have a conclusion with piece-wise smooth integration.

Of course, when $S$ is not a Lagrangian product of bodies in $\mathbb R^n$ and the dual $\mathbb R^n$, but is a smooth symplectic ball with all trajectories closed with action $A$, the argument also applies, giving another proof of Theorem~\ref{theorem:mahler-lph}.

\section{Appendix: Behavior of capacity under reduction}
\label{section:capacity}

Let us check that the proposed symplectic reduction approach is in accordance with the Viterbo conjecture approach to Mahler's conjecture. Consider a symplectic reduction of an arbitrary convex body $S\subset \mathbb R^{2N}$. Compare this with~\cite[Proposition 2.1]{vit2000}, where the displacement energy of a symplectic reduction is estimated from above in a certain way. We assume the reduction linear, this makes things simpler and preserves the convexity of $S$ under the reduction, and this is what we do in the proposed approach to Mahler's conjecture. Let us check how the Ekeland--Hofer--Zehnder capacity of $S$ behaves under a linear reduction to $S' = (S\cap L^\omega)/L$. We also assume that $S$ is smooth and strictly convex. These assumptions are not restrictive once we are aiming at Viterbo's conjecture or other inequalities that allow passing to the limit.

We will pass from convex bodies to norms $\|\cdot\|$, defined as
\[
\|v\| = \sup \{\omega(v,z)\ |\ z\in S\},
\] 
and consider the classical (see~\cite{clarke1979}) variational problem for closed loops $\gamma : \mathbb R/\mathbb Z \to \mathbb R^{2N}$:

\begin{equation}
\label{equation:clarke}
\int_\gamma \|\dot\gamma\| \to \min, \quad \int_\gamma \lambda = 1,
\end{equation}
where $\lambda$ is a primitive of $\omega$. The minimum in this variational problem is the Ekeland--Hofer--Zehnder capacity $c_{EHZ}(S)$ of the convex domain $S$ up to a constant, as shown in~\cite{clarke1979,ekeland1989}.

Now, assume we restrict the minimization problem to those loops $\gamma$ that are contained in a coisotropic linear subspace $L^\omega\subset \mathbb R^{2n}$. For such $\gamma$, the integral $\int_\gamma \lambda$ only depends on the projection of $\gamma$ onto $L^\omega/L$. Assuming that such a projection $\beta$ is given, we can try to restore $\gamma$ by choosing the velocity $\dot\gamma$ as the velocity of smallest norm $\|\dot\gamma\|$ that is projected to the given velocity $\dot\beta$. This just corresponds to restricting the norm $\|\cdot\|$ to $L^\omega$ and then taking the norm on the quotient space by the standard construction.

Of course, in this process of selecting $\dot\gamma$ for given $\dot\beta$ it may happen that the so constructed curve $\gamma$ will not close up, the start and the end points may not match. In general, we have no idea how to handle this issue, but the important particular case of a centrally symmetric $\|\cdot\|$ (needed in Mahler's conjecture) has a remedy \cite{akopyan2018capacity}:

\begin{lemma}[Akopyan, Karasev, 2018]
\label{lemma:symmetry}
In the problem \eqref{equation:clarke}, for centrally symmetric $\|\cdot\|$, one of the minima is attained at a curve $\gamma$ centrally symmetric with respect to the origin. For smooth and strictly convex $S$ we can say more: All minima of \eqref{equation:clarke} are centrally symmetric with respect to some center.
\end{lemma}

We summarize our findings in:

\begin{theorem}
\label{theorem:reduction-capacity}
For a centrally symmetric convex $S\subset\mathbb R^{2n}$, the Ekeland--Hofer--Zehnder capacity cannot decrease in a linear reduction $S' = (S\cap L^\omega)/L$, that is $c_{EHZ}(S') \ge c_{EHZ}(S)$.
\end{theorem}

\begin{proof}
According to Lemma~\ref{lemma:symmetry}, we may assume that the solution of \eqref{equation:clarke} for $S'$, $\beta$, is centrally symmetric. In this case the lift of $\beta$ from $L^\omega/L$ to $L^\omega$ gets closed because its lifted halves may be chosen centrally symmetric to each other. So the capacity will not decrease since going to the reduction corresponds to restricting the domain in the minimization problem.
\end{proof}

\begin{remark}
This theorem gives yet another proof of the main result of \cite{aaok2013}, $c_{EHZ}(K\times K^\circ) \ge 4$ for centrally symmetric $K$, since any such $K\times K^\circ$ can be approximated by linear reductions of $C\times C^\circ$ (cube by cross-polytope), which in turn contains convex symplectic balls of capacity arbitrarily close to $4$ by Proposition~\ref{proposition:convex-lp}.
\end{remark}

\bibliography{../Bib/karasev}

\begin{thebibliography}{10}

\bibitem{abksh2014}
A.~Akopyan, A.~Balitskiy, R.~Karasev, and A.~Sharipova.
\newblock Elementary approach to closed billiard trajectories in asymmetric
  normed spaces.
\newblock {\em Proceedings of the American Mathematical Society},
  144(10):4501--4513, 2016.
\newblock \href{http://arxiv.org/abs/1401.0442}{arXiv:1401.0442}.

\bibitem{ahk2016}
A.~Akopyan, A.~Hubard, and R.~Karasev.
\newblock Lower and upper bounds for the waists of different spaces.
\newblock {\em Topological Methods in Nonlinear Analysis}, 53(2):457--490,
  2019.
\newblock \href{https://arxiv.org/abs/1612.06926}{arXiv:1612.06926}.

\bibitem{akopyan2018capacity}
A.~Akopyan and R.~Karasev.
\newblock Estimating symplectic capacities from lengths of closed curves on the
  unit spheres.
\newblock 2018.
\newblock \href{https://arxiv.org/abs/1801.00242}{arXiv:1801.00242}.

\bibitem{aaok2013}
S.~Artstein-Avidan, R.~Karasev, and Y.~Ostrover.
\newblock From symplectic measurements to the {M}ahler conjecture.
\newblock {\em Duke Mathematical Journal}, 163(11):2003--2022, 2014.
\newblock \href{http://arxiv.org/abs/1303.4197}{arXiv:1303.4197}.

\bibitem{aao2012}
S.~Artstein-Avidan and Y.~Ostrover.
\newblock Bounds for {M}inkowski billiard trajectories in convex bodies.
\newblock {\em International Mathematics Research Notices}, 2014(1):165--193,
  2014.
\newblock \href{http://arxiv.org/abs/1111.2353}{arXiv:1111.2353}.

\bibitem{balitskiy2016}
A.~Balitskiy.
\newblock Shortest closed billiard trajectories in the plane and equality cases
  in {M}ahler's conjecture.
\newblock {\em Geometriae Dedicata}, 184(1):121--134, 2016.

\bibitem{ball1995}
K.~Ball.
\newblock {M}ahler's conjecture and wavelets.
\newblock {\em Discrete and Computational Geometry}, 13:271--277, 1995.

\bibitem{barlov1982}
I.~B\'ar\'any and L.~Lov\'asz.
\newblock Borsuk's theorem and the number of facets of centrally symmetric
  polytopes.
\newblock {\em Acta Mathematica Hungarica}, 40(3--4):323--329, 1982.

\bibitem{cannasdasilva2008}
A.~Cannas~da Silva.
\newblock {\em Lectures on Symplectic Geometry}, volume 1764 of {\em Lecture
  Notes in Mathematics}.
\newblock Springer, 2008.

\bibitem{clarke1979}
F.~Clarke.
\newblock A classical variational principle for periodic {H}amiltonian
  trajectories.
\newblock {\em Proceedings of the American Mathematical Society}, 76:186--188,
  1979.

\bibitem{ekeland1989}
I.~Ekeland and H.~Hofer.
\newblock Symplectic topology and {H}amiltonian dynamics.
\newblock {\em Mathematische Zeitschrift}, 200(3):355--378, 1989.

\bibitem{fhmrz2019}
M.~Fradelizi, A.~Hubard, M.~Meyer, E.~Rold\'an-Pensado, and A.~Zvavitch.
\newblock Equipartitions and {M}ahler volumes of symmetric convex bodies.
\newblock 2019.
\newblock \href{http://arxiv.org/abs/1904.10765}{arXiv:1904.10765}.

\bibitem{grom2003}
M.~Gromov.
\newblock Isoperimetry of waists and concentration of maps.
\newblock {\em Geometric and Functional Analysis}, 13:178--215, 2003.

\bibitem{gruber2007}
P.~M. Gruber.
\newblock {\em Convex and Discrete Geometry}, volume 336 of {\em Grundlehren
  der mathematischen Wissenschaften}.
\newblock Springer Verlag, 2007.

\bibitem{hofer1994}
H.~Hofer and E.~Zehnder.
\newblock {\em Symplectic invariants and {H}amiltonian dynamics}.
\newblock Birkh\"{a}user, 1994.

\bibitem{iriyeh2017}
H.~Iriyeh and M.~Shibata.
\newblock Symmetric {M}ahler's conjecture for the volume product in the three
  dimensional case.
\newblock \href{https://arxiv.org/abs/1706.01749}{arXiv:1706.01749}.

\bibitem{kim2014}
J.~Kim.
\newblock Minimal volume product near {H}anner polytopes.
\newblock {\em Journal of Functional Analysis}, 266(4):2360--2402, 2014.

\bibitem{ku2008}
G.~Kuperberg.
\newblock From the {M}ahler conjecture to {G}auss linking integrals.
\newblock {\em Geometric and Functional Analysis}, 18(3):870--892, 2008.

\bibitem{lmdsch2013}
J.~Latschev, D.~McDuff, and F.~Schlenk.
\newblock The {G}romov width of 4-dimensional tori.
\newblock {\em Geometry \& Topology}, 17:2813--2853, 2013.
\newblock \href{http://arxiv.org/abs/1111.6566}{arXiv:1111.6566}.

\bibitem{lopezreisner1998}
M.~A. Lopez and S.~Reisner.
\newblock A special case of {M}ahler's conjecture.
\newblock {\em Discrete and Computational Geometry}, 20:163--177, 1998.

\bibitem{ma1939}
K.~Mahler.
\newblock Ein \"{U}bertragungsprinzip f\"ur konvexe {K}\"orper.
\newblock {\em \v{C}asopis pro P\u{e}stov\'{a}n\'{\i} Matematiky a Fysiky},
  68:93--102, 1939.

\bibitem{meyerreisner2006}
M.~Meyer and S.~Reisner.
\newblock Shadow systems and volumes of polar convex bodies.
\newblock {\em Mathematika}, 53(1):129--148, 2006.

\bibitem{nazarov2010}
F.~Nazarov, F.~Petrov, D.~Ryabogin, and A.~Zvavitch.
\newblock A remark on the {M}ahler conjecture: Local minimality of the unit
  cube.
\newblock {\em Duke Mathematical Journal}, 154(3):419--430, 2010.

\bibitem{ostrover2014}
Y.~Ostrover.
\newblock When symplectic topology meets {B}anach space geometry.
\newblock 2014.
\newblock \href{https://arxiv.org/abs/1404.6954}{arXiv:1404.6954}.

\bibitem{schlenk2005}
F.~Schlenk.
\newblock {\em Embedding Problems in Symplectic Geometry}, volume~40 of {\em De
  Gruyter Expositions in Mathematics}.
\newblock Berlin, 2005.

\bibitem{vit2000}
C.~Viterbo.
\newblock Metric and isoperimetric problems in symplectic geometry.
\newblock {\em J. Amer. Math. Soc.}, 13(2):411--431, 2000.

\end{thebibliography}
\bibliographystyle{abbrv}
\end{document}